\documentclass{amsart}
\usepackage{amsfonts}
\usepackage{color}
\usepackage{graphicx}
\usepackage{ulem}

\newtheorem{thm}{Theorem}[section]

\newtheorem{lema}[thm]{Lemma}

\theoremstyle{definition}

\theoremstyle{remark}
\newtheorem{rem}[thm]{Remark}

\numberwithin{equation}{section}
\newcommand{\R}{\mathbb R}

\newcommand{\C}{\mathcal{C}}

\newcommand{\A}{\mathcal{A}}

\newcommand{\ve}{\varepsilon}
\newcommand{\lam}{\lambda}

\begin{document}
\title[A Lyapunov type Inequality]{A Lyapunov type Inequality for  Indefinite Weights and Eigenvalue Homogenization}
\author[J Fern\'andez Bonder, J P Pinasco, A M Salort]{Juli\'an Fern\'andez Bonder, Juan P. Pinasco, Ariel M. Salort }
\address{Departamento de Matem\'atica \hfill\break \indent   IMAS - CONICET
 \hfill\break \indent FCEyN - Universidad de Buenos Aires
\hfill\break \indent Ciudad Universitaria, Pabell\'on I \hfill\break \indent   (1428)
Av. Cantilo s/n. \hfill\break \indent Buenos Aires, Argentina.}
\email[J. Fernandez Bonder]{jfbonder@dm.uba.ar}
\urladdr[J. Fernandez Bonder]{http://mate.dm.uba.ar/~jfbonder}
\email[J.P. Pinasco]{jpinasco@dm.uba.ar}
\urladdr[J.P. Pinasco]{http://mate.dm.uba.ar/~jpinasco}
\email[A.M. Salort]{asalort@dm.uba.ar}

\begin{abstract}
In this paper we prove a Lyapunov type inequality for quasilinear problems with indefinite weights. We show that
the first eigenvalue is bounded below in terms of the integral of the weight, instead of the integral of its
 positive part. We apply this inequality to some eigenvalue homogenization problems with indefinite weights.
 \end{abstract}

\maketitle

\section{Introduction}

In this paper we will consider the following quasilinear eigenvalue problem,
\begin{equation}\label{one}
 \begin{cases}
   -( a(x)|u'|^{p-2}u')'= \lam \rho(x)|u|^{p-2}u  \qquad x\in (0,L)\\
   u(0)=u(L)=0,
 \end{cases}
\end{equation}
where $1<p<\infty$, $\lambda$ is a real parameter, the weight $\rho \in L^1(0,L)$ is allowed to
change signs, and the coefficient $a(x)$ belongs to the Muckenhoupt
class $ \A_p$, see Section \S 2 for definitions.

Quasilinear eigenvalue problems with indefinite weighs were studied first by Zeidler in \cite{Zeidler}, with $a\equiv 1$, and the extension
to general coefficients  $a\in \A_p$ follows by using the techniques in \cite{OpK} and \cite{Ture}.
There are several results concerning the existence and asymptotic behavior of   two sequences of eigenvalues
$\{\lam_{k}^+\}_{k\ge 1}$, $\{\lam_{k}^-\}_{k\ge 1}$,  going to $\pm \infty$, in both the linear and nonlinear case,  in one and several dimensions,
 see for instance \cite{All, Cuesta, FBP, Fl, Kato, Sm}.

For one dimensional problems, the eigenvalues are simple, in the sense that $\lam_k^{\pm}$ has a unique (up to a
multiplicative constant) associated eigenfunction $u_k^{\pm}$ with exactly $k+1$ zeros
in $[0,L]$. When $a\equiv 1$, the asymptotic behavior of the eigenvalues is given by
\begin{equation}\label{asymp}
\lambda_k^{\pm} \sim\left(\frac{\pi_p}{\int_0^L (\rho^{\pm}(x))^{1/p}
dx}\right)^p,
\end{equation}
where $\rho^+(x) = \max\{\rho(x), 0\}$, and $\rho^-=(-\rho)^+$, and
$$
 \pi_p =  2(p-1)^{1/p} \int_0^1 \frac{ds}{(1-s^p)^{1/p}}.
$$
The case of nonconstant
principal coefficient can be analyzed easily in exactly the same way by using a change of
variable (see Section \S 2).

However,   the asymptotic formula
\eqref{asymp} cannot be used to derive precise bounds for $\lam_k^{\pm}$.
A useful tool in order to obtain lower bounds for eigenvalues in one dimensional problems is Lyapunov's inequality: if
there exists a nontrivial solution $u$ of
 $$ \begin{cases}
   -u''= \rho(x)u  \qquad x\in (0,L)\\
   u(0)=u(L)=0,
 \end{cases}$$
then we must have
$$
\frac{4}{ L } \le   \int_0^L |\rho(t)|dt,
$$
an inequality proved by Borg \cite{Bo1}, and improved later by Wintner \cite{Wi} who
used $\rho^+(t)$ instead of $ |\rho(t)|$. A similar inequality holds for $p-$Laplacian
operators, and applying it in each nodal domain, we can bound the $k-$th eigenvalue as follows,
$$
\frac{2^p k^p}{ L^{p-1}} \le \lam_k   \int_0^L \rho^+(t)dt,
$$
see \cite{Pi}, and \cite{Pi1} for a recent survey.

\bigskip
However, it is possible to obtain better bounds of the eigenvalues, and we will prove the
following Lyapunov type inequality which takes care of the negative part of the weight
as well:

\begin{thm}\label{lyapi}
Let $u$ be a solution of
 \begin{equation}\label{plapuna}
 -( a(x)|u'|^{p-2}u')'=  \rho(x)|u|^{p-2}u  \qquad x\in (0,L)
 \end{equation}
satisfying the boundary condition $u(0)=u(L)=0$. Then,
\begin{equation}\label{lyafirst}
\frac{1}{p}\left(\int_0^L  a^{-\frac{1}{p-1}}(t)
   dt\right)^{1-p}  \le     \sup_{0\le x\le L} \left|\int_0^x \rho(t)dt \right|
 \end{equation}
\end{thm}

As a consequence, we have the following bound for $\lam_k^+$:

\begin{thm}\label{lyapu}
Let $\lambda_k^+$ be the $k$-th eigenvalue of
problem \eqref{one}  with $\rho \in L^1[0,L]$. Then,
\begin{equation}\label{jfbjppams}
 \frac{k^{p-1}}{p} \left(  \int_{0}^{L}  a^{-\frac{1}{p-1}}(t)
   dt  \right)^{1-p}  \le \;  \lam_k \cdot \sup_{(a,b)\subset [0,L]} \left|\int_a^b \rho(t)dt \right|
.
\end{equation}
\end{thm}

Let us mention that inequalities \eqref{lyafirst} and \eqref{jfbjppams} are closely
related to the ones introduced by Harris and Kong in \cite{HaKo} and generalized in
\cite{Hong} for quasilinear problems like \eqref{one} with $a\equiv 1$. However, a mixed boundary
condition was involved in these inequalities: suppose that $u'(0)=u(L)=0$, then
$$
1 \le L^{p-1} \sup_{0\le x\le L} \int_0^x \rho(t)dt,
$$
and, for $u(0)=u'(L)=0$,
$$
1 \le L^{p-1} \sup_{0\le x\le L} \int_x^L \rho(t)dt.
$$
These inequalities were proved by using Ricatti equation techniques, that is, by
introducing the function $v=u'/u$, and studying  a first order nonlinear differential
equation for $v$, with the initial condition $v(0)=0$ and using the fact that $v$
blows up when $x\to L$.

The main drawback of this approach is that it cannot be used for higher order
equations, so our proofs of Theorems \ref{lyapi}, \ref{lyapu} and \ref{lyapi2} are
based on variational arguments,  and we give the following extension  to higher order differential equations:

\begin{thm}\label{lyapi2}
Let $u$ be a solution of
\begin{equation}\label{oness}
 \begin{cases}
    (-1)^{m}u^{(2m)} =  \rho(x) u \qquad x\in (0,L)\\
  u^{(j)}(0)=u^{(j)}(L)=0 \quad 0\le j\le m,
 \end{cases}
\end{equation}
where $m\ge 1$ and  $\rho \in L^1[0,L]$. Then,
\begin{equation}\label{jfbjppams2}
 \frac{ (m-1) [(m-2)!]^2}{ 2 L^{2m-1} }  \le   \sup_{0\le x\le L}\int_0^x \rho(t)dt.
\end{equation}
\end{thm}

With the same ideas we will prove a version of Theorem \ref{lyapi2} including higher order quasilinear equations.

\begin{rem}
When $p=2$, $a\equiv 1$, and $\rho\ge 0$, the constant obtained by Das and
Vatsala in \cite{Das} for the classical Lyapunov's inequality is given by
$$
\frac{2 \, 4^{2m-1}(2m-1) [(m-2)!]^2 }{2 L^{2m-1}  }.
$$
For $p\neq 2$ and $m>1$ the determination of the best constant remains open even for the classical Lyapunov's inequality with positive
weights, see \cite{Wata}.

Although the constant in Theorem \ref{lyapi2} is worse than the one in the classical inequality, we will see in the applications that
 the cancelation of the
positive and negative parts of the weight $\rho$  in the integral gives much better bounds in several
cases. Hence,  the true contribution of this kind of inequalities cames from the
right-hand side.
\end{rem}

\bigskip
Let us focus now in applications of Theorem  \ref{lyapu}. Let us
consider first the following periodic homogenization problem
\begin{equation} \label{P1}
 \begin{cases}
   -(a(\tfrac{x}{\ve})|u_\ve'|^{p-2}u_\ve')'=\lam_{\ve} \rho(\tfrac{x}{\ve})|u_\ve|^{p-2}u_\ve
      \quad x\in (0,L)\\
   u_\ve(0)=u_\ve(L)=0,
 \end{cases}
\end{equation}
where $\ve$ is a positive  small parameter, and we assume that $a$, $\rho$ are $L$-periodic functions.  For each fixed $\ve$, the results for
problem \eqref{one} hold:  we have two sequences of simple eigenvalues
$\{\lam_{\ve,k}^{\pm}\}_{k\ge 1}$,  and their corresponding eigenfunctions satisfy the
same nodal properties as before.

  We are interested in the behavior of
$\{\lam_{\ve,k}^{\pm}\}_{\ve}$ when $\ve \to 0^+$ for each $k$ fixed. For nonnegative
weights $\rho$, it is known (see \cite{Ch, FBPS2, FBPS1}) that $\lam_{\ve, k} \to \lam_{0,k}$
when $\ve \to 0^+$, where $\lam_{0,k}$  is the $k-$th eigenvalue of the limit problem
 \begin{equation}\label{limit}
 \begin{cases}
   -(a^*|u'|^{p-2}u')'=\lam_0 \bar{\rho} |u|^{p-2}u
      \qquad \qquad x\in (0,L)\\
   u(0)=u(L)=0,
 \end{cases}
 \end{equation}
where $\bar\rho$ is the average of $\rho$ in the interval $(0,L)$ and $a^*$ is given
by
$$
a^* := \left(\int_0^L a(t)^{-\frac{1}{p-1}}\, dt\right)^{1-p },
$$
see \cite{FBPS1} for a proof. Remarkably, observe that  $a^*$ is the constant appearing in Lyapunov's inequality
\eqref{lyafirst}.

We may ask what happens when $\rho$ is an indefinite weight, and our next Theorem
generalizes to the one dimensional quasilinear setting the answer for second order
linear problems (in arbitrary spatial dimension) obtained recently by Nazarov, Pankratova and Piatnitski in \cite{Naz, Pank}:
\begin{thm}\label{main}
Let $\{\lam_{\ve, k}^{\pm}\}_{k\ge 1}$ be the eigenvalues of problem
\eqref{P1}, and $\{\lam_{0,k}\}_{k\ge 1}$ be the eigenvalues of problem \eqref{limit}.
Then,
\begin{enumerate}

\item If $\bar\rho = 0$  then $\lam_{\ve,k}^{\pm} \to \pm\infty$ as $\ve\to 0^+$.

\item If $\bar\rho < 0$  then $\lam_{\ve,k}^- \to \lam_{0,k}$ and  $\lam_{\ve,k}^+
    \to +\infty$ as $\ve\to 0^+$.

\item If $\bar\rho > 0$  then $\lam_{\ve,k}^+ \to \lam_{0,k}$ and  $\lam_{\ve,k}^-
    \to -\infty$ as $\ve\to 0^+$.
\end{enumerate}
\end{thm}

Let us remark that in the first case, after a suitable renormalization $\mu_{\ve,
k}^{\pm} = \ve^{\alpha}\lam_{\ve,k}^{\pm}$, the convergence to the eigenvalues of a
different limit problem was obtained in \cite{Naz, Pank}.  Their proofs were based on  linear tools such
as orthogonality of eigenfunctions or asymptotic expansions in powers of $\ve$ which are not available here.

\bigskip

The paper is organized as follows: In Section \S 2 we introduce some necessary facts
about the eigenvalue problem for the $p-$Laplace operator. In Section \S 3 we prove
Theorems \ref{lyapi}, \ref{lyapu} and \ref{lyapi2}.  Section \S 4 is devoted to the proof of Theorem
\ref{main}.

\section{Preliminary results}

\subsection{Eigenvalues of $p$-Laplacian operators}\label{epl}

The eigenvalue problem for the $p$-Laplacian operator started with the
pioneering results of Browder in the 1960s. A recent survey including more general
quasilinear homogeneous operators in $\Omega \subset R^N$ can be found in
\cite{FBPS1}. We will need here the following variational characterization of
eigenvalues:

\begin{thm} There exists a double sequence of   variational eigenvalues
 $\{\lam_k^{\pm}\}_{k\ge 1}$ of problem \eqref{one}, given by
\begin{equation}\label{variac}
\lam_k^{\pm} = \inf_{C \in \C_k}  \sup_{u \in C} \int_0^L a(x)|u'|^p\, dx
\end{equation}
where
\begin{align*}
\C_k & = \{ C \subset M^{\pm} \  : \  C \  \mbox{ is compact,} \ C=-C, \
\gamma(C) \ge k \},\\
M^{\pm} & = \left\{u \in W^{1,p}_0(0,L) \ : \ \int_0^L \rho(x) |u|^p\, dx = \pm 1 \right\}.
\end{align*}
and $\gamma$ is the Krasnoselskii genus.
\end{thm}

The proof is identical to the one in \cite{GAP} for positive weights.

\begin{rem}
Sometimes we will work with an equivalent definition (due to the homogeneity of the
operator),
\begin{equation}\label{rayleigh}
 \lam_k^\pm = \inf_{C \in \C_k}  \sup_{u \in C} \frac{ \int_0^L
a(x)|u'|^p\, dx}{\int_0^L \rho(x) |u|^p },
\end{equation}
where now
 $$M^+  = \left\{u \in W^{1,p}_0(0,L) \ : \ \int_0^L \rho(x) |u|^p\, dx >0 \right\},$$
 $$M^-  = \left\{u \in W^{1,p}_0(0,L) \ : \ \int_0^L \rho(x) |u|^p\, dx <0 \right\}.$$
\end{rem}

We need also the following results from the Sturm-Liouville theory for the
$p-$Laplacian:

\begin{thm}\label{zeros} Let $\{u_k^{\pm}\}_{k\ge 1}$ be a sequence of normalized eigenfunctions corresponding to
$\{\lam_k^{\pm}\}_{k\ge 1}$. Then, $u_k^{\pm}$ has exactly $k+1$ zeros in $[0,L]$
which determine $k$ nodal domains.
\end{thm}

\begin{thm}\label{pesos} Let $\rho_1$, $\rho_2 \in L^{1}(0,L)$, with $\rho_1(x) \le \rho_2(x)$. Then,
$$\lam_k^+(\rho_1) \ge \lam_k^+(\rho_2),$$
$$\lam_k^-( \rho_1) \le \lam_k^-(\rho_2),$$
whenever $\lam_k^\pm(\rho_1)$ and $\lam_k^\pm(\rho_2)$ exists.
\end{thm}

  The proof is a consequence of the alternative variational
characterization \eqref{rayleigh}), by observing that $M^+(\rho_2)\subset
M^+(\rho_1)$, and after changing $\rho \to -\rho$ the other follows as well.
Clearly, if one weight is negative, and the other one is positive, there are no
eigenvalues to compare.

\subsection{The Muckenhoupt class of $A_p$ weights}\label{apw}

We say that $a$ belongs to the Muckenhoupt $\A_p$-class if $a$ is a nonnegative
function in $L^1_{loc}(\R^N)$, such that
\begin{equation}\label{apwe}
 \left( \int_B a(x)dx\right)\left( \int_B a(x)^{-\frac{1}{p-1}}dx\right)^{p-1} \le c_{p,a}|B|^p
\end{equation}
for any ball $B\in \R^N$, and a fixed positive constant $c_{p,a}$.

\begin{rem}
Beside the fact that the bounds we will study below depend strongly on the existence
of the second integral in equation \eqref{apwe}, the most important fact is that
$\A_p$ weights guarantee the density of smooth functions in $W^{1,p}(\Omega)$, see
\cite{Kil, Ture}, and is the right class of coefficients involved in different quasilinear problems.
The $\A_p$ class contains several interesting weights, like powers of the distance to
a point or the distance to the boundary, allowing both singular and degenerate
coefficients.
\end{rem}

We can avoid the coefficient $a$ in problem \eqref{one} by performing a change of variables.  If we define
$$
P(x) = \int_0^x \frac{1}{a(s)^{1/(p-1)}} ds,
$$
and perform the change of variables
$$
(x,u)\to (y, v)
$$
where
$$
y = P(x), \qquad v(y)= u(x),
$$
a simple computation gives
$$
\begin{cases}
 -(|\dot{v}|^{p-2}\dot{v})^{\cdot} =  \lam Q(y) |v|^{p-2} v, &  y\in [0,\ell]\\
 v(0)=v(\ell)=0
 \end{cases}
$$
where
$$\quad\cdot =    d/dy,$$
with
$$
\ell = \int_0^1 \frac{1}{a(s)^{1/(p-1)}} ds \to \ell = \overline{a^{\frac{-1}{p-1}}},
$$
and
\begin{align*}
Q(y)& =  a(x)^{1/(p-1)} \rho(x) \\
& =   a(P^{-1}(y))^{1/(p-1)}\rho(P^{-1}(y)).
\end{align*}

Both $P$ and $\ell$ are well defined whenever $a\in \A_p$.

\section{A Lyapunov type inequality}

We prove in this section our main results. Before the proofs, let us state and prove
the following useful lemma:

\begin{lema}\label{taylorbound} Let $1<p<\infty$, and let
 $u \in C^{(m)}_0(0,L)$,  with
 $$
u^{j}(a)= 0, \qquad 0\le j \le m-1,
$$
and $a(x)\in A_p$. Then,
$$
\|u \|_{L^{\infty}(0,L)} \le \frac{L^{m-1}}{(m-1)!} \left(\int_0^L  a^{-\frac{1}{p-1}}(t)
   dt\right)^{\frac{p-1}{p}} \left(\int_0^L  a(t)  \left|   u^{(m)}(t)  \right|^p
   dt\right)^{\frac{1}{p}}.
$$
\end{lema}

\begin{proof}
We apply  Taylor's expansion and the Cauchy expression for the remainder, obtaining
 $$
 u(x)= \sum_{j=0}^{m-1} \frac{u^{j}(0)}{j!}x^j+\int_0^x \frac{(x-t)^{m-1}}{(m-1)!}u^{(m)}(t)dt.$$
 However, since the  derivatives of $u$ at $0$ are zero, we have
 $$
 u(x)= \int_0^x \frac{(x-t)^{m-1}}{(m-1)!}u^{m}(t)dt.$$

  Thus, bounding $x-t<L$, and by using Holder's inequality,
  $$\begin{array}{rcl}
 \displaystyle  |u(x)| & \le &  \displaystyle L^{m-1}   \int_0^x  a^{\frac{1}{p}}(x)a^{-\frac{1}{p}}(x) \frac{ u^{(m)}(t) }{ (m-1)! }
   dt \\ \\
   & \le &  \displaystyle L^{m-1} \left(\int_0^x  a^{-\frac{1}{p-1}}(x)
   dt\right)^{\frac{p-1}{p}} \left(\int_0^x  a(x)  \left| \frac{ u^{(m)}(t) }{ (m-1)! }\right|^p
   dt\right)^{\frac{1}{p}},
 \end{array}$$
and the inequality is proved.
\end{proof}

\begin{proof}[Proof of Theorem \ref{lyapi}]
By multiplying equation \eqref{plapuna} by $u$, and integrating by parts we get
$$
 \int_0^L a(x) |u^{\prime}|^p dx = \int_0^L \rho(x)|u|^p dx.
$$

Let us introduce now the function $Q(x)=\int_0^x \rho(t)dt$. Hence,
 $$ \int_0^L \rho(x)|u|^p dx= \int_0^L Q'(x)|u|^p dx = -\int_0^L Q(x) p |u|^{p-2}uu'dx.$$

 By combining these with Lemma \ref{taylorbound}, for $m=1$, we obtain
 $$\begin{array}{rcl}
 \displaystyle
 \int_0^L a(x) |u^{\prime}|^p dx & = &
 \displaystyle \int_0^L \rho(x)|u|^p dx \\ \\
   &\le & \displaystyle
 p  \sup_{0\le x\le
L} |Q| \|u\|_{L^\infty(0,L)}^{p-1} \int_0^L  |u'| dx
  \\ \\
    &\le & \displaystyle
 p  \sup_{0\le x\le
L} |Q|  \left(\int_0^L  a^{-\frac{1}{p-1}}(t)
   dt\right)^{\frac{(p-1)^2}{p}} \\ \\
    & & \displaystyle \qquad \times \left(\int_0^L  a(t)  |   u^{\prime}(t) |^p
   dt\right)^{\frac{p-1}{p}}  \int_0^L  |u'| dx.
   \end{array}$$

Now, by inserting $a^{\frac{1}{p}}(x)a^{-\frac{1}{p}}(x)$ in the last integral, and
using again Holder's  inequality, we get after rearranging,
$$
 \int_0^L a(x) |u^{\prime}|^p dx  \le  \displaystyle
 p  \sup_{0\le x\le
L} |Q|  \left(\int_0^L  a^{-\frac{1}{p-1}}(t)
   dt\right)^{p-1}  \int_0^L  a(t) |  u^{\prime}(t)  |^p
   dt
   $$
and, after canceling the integral in both sides, the theorem is proved.
 \end{proof}

In the next proof we use that the eigenfunction $u_k$ associated to $\lam_k$ has
exactly $k$ nodal domains.

\begin{proof}[Proof of Theorem \ref{lyapu}] Let us fix $k\ge 1$, and let $0=x_0<x_1 < \cdots< x_n=L$ be
the $k+1$ zeros of the associated eigenfunction $u_k$. We can apply Lemma
\ref{plapuna} in each nodal domain $(x_{i-1},x_i)$, obtaining
$$
\frac{1}{p^{\frac{1}{p-1}}} \le  \lam_k^{\frac{1}{p-1}}  \left(  \sup_{x_{i-1}\le x\le x_i} \left|\int_{x_{i-1}}^x   \rho(t)dt \right|
\right)^{\frac{1}{p-1}} \cdot
 \int_{x_{i-1}}^{x_{i}}  a^{-\frac{1}{p-1}}(t)
   dt,   $$
and therefore,
 \begin{align*} \frac{ k}{p^{\frac{1}{p-1}}}  &  \le  \lam_k^{\frac{1}{p-1}}  \sum_{i=1}^k
    \left( \sup_{x_{i-1}\le x\le x_i}  \left|\int_{x_{i-1}}^x   \rho(t)dt \right|\right)^{\frac{1}{p-1}}
    \cdot
 \int_{x_{i-1}}^{x_{i}}  a^{-\frac{1}{p-1}}(t)
   dt \\
    & \le  \lam_k^{\frac{1}{p-1}}   \left(   \sup_{(a,b)\subset [0,L]} \left|\int_a^b  \rho(t)dt \right|\right)^{\frac{1}{p-1}}
    \cdot
      \sum_{i=1}^k
 \int_{x_{i-1}}^{x_{i}}  a^{-\frac{1}{p-1}}(t)
   dt. \\
\end{align*}
Hence,
 $$ \frac{1}{p} \left(\frac{k}{ \int_{0}^{L}  a^{-\frac{1}{p-1}}(x)
   dt} \right)^{p-1}  \le \lam_k   \sup_{(a,b)\subset [0,L]} \left|\int_a^b \rho(t)dt \right|
$$
and the proof is finished
\end{proof}

The proof of Theorem \ref{lyapi2} is similar. Let us prove
the following result which is more general:

\begin{thm}\label{lyapimp}
Let $u \in W_0^{m,p}(0,L)$ be a nontrivial solution  of
\begin{equation}\label{onemp}
 \begin{cases}
    (-1)^{m}(a(x)|u^{(m)}|^{p-2}u^{(m)})^{(m)} = \rho(x) u \qquad x\in (0,L)\\
  u^{(j)}(0)=u^{(j)}(L)=0 \quad 0\le j\le m,
 \end{cases}
\end{equation}
where $m\ge 1$, $a\in \A_p$,  and  $\rho \in L^1$. Then
\begin{equation}\label{jfbjppams3}
 \frac{ (m-1)^{p-1} [(m-2)!]^p}{ p L^{mp-p} } \left(\int_0^L a^{-\frac{1}{p-1}}(x)
   dt\right)^{1-p} \le   \sup_{0\le x\le L}\int_0^x \rho(t)dt.
\end{equation}
\end{thm}

\begin{proof} We start from
 \begin{align*}
 \int_0^L a(x) |u^{(m)} |^p dx & = \int_0^L  Q(x)'|u|^p dx \\
 & = -\int_0^L Q(x) p |u|^{p-2}uu'dx \\
 & \le  p  \sup_{0\le x\le
L} |Q|   \int_0^L   \left| u \right|^{p-1} \left| u' \right| dx \\
   \end{align*}
with $Q=\int_0^x \rho(t)dt$ as before.

Let us call $A  = \| u \|_{L^{\infty}}^{p-1} $ and $ B   = \| u' \|_{L^{\infty}} $.
Since  $u\in W^{m,p}_0$ implies that $u'\in W^{m-1,p}_0$, using Lemma
\ref{taylorbound}, we have
 \begin{align*}
A &  \le
\left[\frac{L^{m-1}}{(m-1)!}\left(\int_0^L  a^{-\frac{1}{p-1}}(x)
   dt\right)^{\frac{p-1}{p}} \left(\int_0^L  a(x)  \left|   u^{(m)}(t)  \right|^p
   dt\right)^{\frac{1}{p}} \right]^{p-1}.    \\
   \\   B   & \le
   \frac{L^{m-2}}{(m-2)!} \left(\int_0^L  a^{-\frac{1}{p-1}}(x)
   dt\right)^{\frac{p-1}{p}} \left(\int_0^L  a(x)  \left|   u^{(m)}(t)  \right|^p
   dt\right)^{\frac{1}{p}}.
   \end{align*}
Now, we can bound
  \begin{align*}
  \int_0^L   \left| u \right|^{p-1} \left| u' \right| dx  &   \le
A\cdot B\cdot L \\ \\
 & = C \left(\int_0^L a^{-\frac{1}{p-1}}(x)
   dt\right)^{p-1}  \int_0^L  a(x)  \left|   u^{(m)}(t)  \right|^p
   dt \end{align*}
   where
$$ C  =  \frac{ L^{mp-p} }{ (m-1)^{p-1} [(m-2)!]^p},
  $$
and the result follows.
 \end{proof}

\begin{rem}
When $p=2$, and $a\equiv 1$, we get
 $$
 \frac{ (m-1) [(m-2)!]^2}{ 2 L^{2m-1} }   \le   \sup_{0\le x\le L}\int_0^x \rho(t)dt
 $$
as in Theorem \ref{lyapi2}.
\end{rem}

\section{The homogenization problem}

\begin{proof}[Proof of Theorem \ref{main}.]

We prove only the assertions for the positive eigenvalues $\{\lam_{\ve,k}^+\}$, for the negative ones the result follows by considering the weight $-\rho$.

(1)
Let $\rho \in L^1 $ be an $L$-periodic function with zero mean.
Then, a simple computation gives.
$$
\sup_{(a,b)\subset [0,L]} \left|\int_a^b \rho(\tfrac{x}{\ve})dt \right|\le
\sup_{(a,b)\subset [0,\ve]}  \int_a^b |\rho(\tfrac{x}{\ve})|dt \le \ve \|\rho\|_{L^1(0,L)}.
$$

Now, Lyapunov's inequality \eqref{jfbjppams} gives
$$
\frac{k^{p-1}}{\ve p  \|\rho\|_{L^1(0,L)}} \left(  \int_{0}^{L}  a^{-\frac{1}{p-1}}(t)
   dt  \right)^{1-p}  \le \;  \lam_{\ve,k}^+
$$
which proves that the eigenvalues diverge as $\ve\to 0^+$.

\bigskip
(2)  Let us assume that $\bar\rho<0$, and let us introduce the weight
 $$
 \sigma = \rho - \bar\rho.$$
Now, $\sigma > \rho$, and $\bar\sigma=0$. So, from from the Sturmian comparison theorem \ref{pesos}, and part (1) we get
$$
\lam_{\ve,k}^+(\rho)\ge\lam_{\ve,k}^+(\sigma) \ge
\frac{k^{p-1}}{\ve p  \|\sigma\|_{L^1(0,L)}} \left(  \int_{0}^{L}  a^{-\frac{1}{p-1}}(t)
   dt  \right)^{1-p},
$$
and the result follows.

\bigskip
(3) Let us assume now that $\bar\rho>0$, and let us show that the eigenvalues $\lam_{\ve,k}^+$ convege to $\lam_{0,k}^+$. To this end, we need to show first
that $\{\lam_{\ve,k}^+\}_{\ve}$ is uniformly bounded away from zero and infinity. Then, the remaining steps of the proof are the same as in the case of positive
weights, and can be found
in \cite{Ch,FBPS1}:  we extract a convergent sequence of eigenvalues, we choose a sequence of normalized eigenfunctions corresponding to these
eigenfunctions, and Rellich-Kondrachov enable us to extract a convergent subsequence.
The convergence of the full family follows from the simplicity of the eigenfuctions of problem \eqref{limit}, see \cite{FBPS1} for details.

So, we only need to show that $c <\lam_{\ve,k}^+< C$ for some positive constants $c$, $C$.

The lower bound follows from Lyapunov's inequality \eqref{jfbjppams},
$$
 \frac{k^{p-1}}{p} \left(  \int_{0}^{L}  a^{-\frac{1}{p-1}}(\tfrac{t}{\ve})
   dt  \right)^{1-p}  \le \;  \lam_{\ve,k}^+ \cdot \sup_{(a,b)\subset [0,L]} \left|\int_a^b \rho(\tfrac{t}{\ve})dt \right|
\le \lam_{\ve,k}^+  \int_0^L |\rho(\tfrac{t}{\ve})|dt,
$$
and using that
$$\int_0^L f(\tfrac{t}{\ve})dt = \int_0^L f(t)dt + O(\ve)$$
for any $L$-periodic function $f \in L^1$, we get a lower bound for $\lam_{\ve,k}^+$ independently of $\ve$.

Let us find an upper bound of the first eigenvalue. To this end, we construct a test function and compute the Rayleigh  quotient \eqref{rayleigh}. We choose
$h$ small,
and we define a continuous function $w$ which grows linearly from zero to 1 in $[0,h]$, remains constant in $[h,L-h]$, and decreases linearly from 1 to zero in $[L-h,L]$.
Now, we have
\begin{align*}
\int_0^L a(\tfrac{t}{\ve})  |w'|^p & \le 2h^{-p} \int_0^L a(\tfrac{t}{\ve})dt  \\
 & \le 4L h^{-p} \|a\|_{L^1([0,L])} \\ \\
 \int_0^L \rho(\tfrac{x}{\ve}) |w|^p dx & \ge \left(\Big\lfloor \frac{L-2h}{\ve}\Big\rfloor-2\right) \ve L\int_0^L \rho(t) dt
- 2 \left(\Big\lfloor \frac{h}{\ve}\Big\rfloor+2\right) \ve L\int_0^L |\rho(t)| dt \\
& \ge \frac{L^3}{2} \bar\rho,
\end{align*}
for $h$ fixed sufficiently small and any $\ve<\ve_0(h)$.

Therefore,
$$
\lam_{\ve,1}^+ \le \frac{8 \|a\|_{L^1([0,L])} }{h^p L^2\bar\rho},
$$
and the first eigenvalue is bounded above.

The same argument can be applied for higher eigenvalues. In fact let $w$ be the previously defined function and introduce $k$ functions $\{w_i\}_{1\le i \le k}$ defined on $[0,L]$, where
$$
w_i(x)= \left\{\begin{array}{ll} w(kx-(i-1)L) & x\in [(i-1)L/k, iL/k], \\
0 & x\in [0,L]\setminus [(i-1)L/k, iL/k].\end{array}\right.
$$

We consider the set $C=span\{w_1,\dots, w_k\} \cap S$,  where $S$ is the unit ball in $W_0^{1,p}([0,L])$. Since $\gamma(C)=k$, and $C\subset M^+(\rho)$ for $\ve<\ve_0(k)$, it is an
admissible set in the characterization of $\lam_{\ve,k}^+$. Thus, if $v=\sum c_i w_i\in C$,
we have $\sum_{i=1}^k |c_i|^p = 1$ and
\begin{align*}
\int_0^L a(\tfrac{x}{\ve})  |v'|^p\, dx & \le 2h^{-p}k^p \int_0^L a(\tfrac{x}{\ve})\, dx \sum_{i=1}^k |c_i|^p  \\
 & \le 4L h^{-p} k^p\|a\|_{L^1([0,L])} \\ \\
\int_0^L \rho(\tfrac{x}{\ve}) |v|^p \, dx & = \sum_{i=1}^k |c_i|^p \int_0^L \rho(\tfrac{x}{\ve}) |w_i|^p \, dx
\\  & \ge \sum_{i=1}^k |c_i|^p  \frac{L^3}{k^3 2} \bar\rho \\
& =    \frac{L^3}{2k^3} \bar\rho .
\end{align*}
and we get
$$
 \lam_{\ve,k}^+ = \inf_{C \in \C_k}  \sup_{v \in C} \frac{ \int_0^L
a(\tfrac{x}{\ve})|v'|^p\, dx}{\int_0^L \rho(\tfrac{x}{\ve}) |v|^p \, dx} \le \frac{8 k^{p+3}}{L^2 \bar\rho} ,
$$
and the proof is finished.
\end{proof}

\section*{Acknowledgements}

This work was partially supported by Universidad de Buenos Aires under grant  20020130100283BA, ANPCyT PICT2012 0153,
and by CONICET (Argentina) PIP 5478/1438.

\def\cprime{$'$}
\providecommand{\bysame}{\leavevmode\hbox to3em{\hrulefill}\thinspace}
\providecommand{\MR}{\relax\ifhmode\unskip\space\fi MR }
\providecommand{\MRhref}[2]{%
  \href{http://www.ams.org/mathscinet-getitem?mr=#1}{#2}
}
\providecommand{\href}[2]{#2}

\end{document}